\documentclass{amsart}
\usepackage{amssymb}
\usepackage{amsthm}
\usepackage{pb-diagram}
\usepackage{color}

\newtheorem{theorem}{Theorem}[section]

\newtheorem{prop}[theorem]{Proposition}

\newtheorem{fact}[theorem]{Fact}
\newtheorem{cor}[theorem]{Corollary}
\newtheorem{proposition}[theorem]{Proposition}
\newtheorem{lemma}[theorem]{Lemma}
\newtheorem{corollary}[theorem]{Corollary}

\newtheorem*{question*}{Question}

\theoremstyle{definition}
\newtheorem{definition}[theorem]{Definition}

\newtheorem{question+}[theorem]{Question}

\newtheorem{defn}[theorem]{Definition}

 \newtheorem{example}[theorem]{Example}

\theoremstyle{remark}
\newtheorem{rmk}[theorem]{Remark}
\newtheorem{remark}[theorem]{Remark}

 \DeclareMathOperator{\intr}{int}
  \DeclareMathOperator{\bd}{bd}

\newcommand{\WR}{\widetilde{\cal R}}

\newcommand{\la}{\langle}
\newcommand{\ra}{\rangle}

\newcommand{\sub}{\subseteq}

\newcommand{\cal}[1]{\ensuremath{\mathcal{#1}}}

\newcommand{\Rarr}{\ensuremath{\Rightarrow}}

\newcommand{\res}{\ensuremath{\upharpoonright}}

\newcommand{\cl}[1]{\begin{overline}{#1}\end{overline}}

\newcommand{\lam}{\ensuremath{\lambda}}

\newcommand{\es}{\ensuremath{\emptyset}}

\newcommand{\sm}{\setminus}

\newcommand{\Z}{\mathbb{Z}}
\newcommand{\N}{\mathbb{N}}
\newcommand{\Q}{\mathbb{Q}}
\newcommand{\R}{\mathbb{R}}

\title[Expansions with no new smooth functions]
{Expansions of real closed fields which introduce no new smooth functions}

\begin{document}

\author {Pantelis  E. Eleftheriou}

\author {Alex Savatovsky}

\address{Department of Mathematics and Statistics, University of Konstanz, Box 216, 78457 Konstanz, Germany}

\email{panteleimon.eleftheriou@uni-konstanz.de, alex.savatovsky@uni-konstanz.de}

\thanks{Research supported by an Independent Research Grant from the German Research Foundation (DFG), Zukunftskolleg Research Fellowship and a Young Scholar Fund.}


\date{\today}
\begin{abstract} We prove the following theorem: let $\WR$ be an expansion of the real field $\overline \R$, such that every definable set (I) is a uniform countable union of semialgebraic sets, and (II)  contains a ``semialgebraic chunk".
Then every definable smooth function $f:X\sub \R^n\to \R$ with open semialgebraic domain is semialgebraic.

Conditions (I) and (II) hold for various d-minimal expansions $\WR = \la \overline \R, P\ra$ of the real field, such as when $P=2^\Z$, or $P\sub \R$ is an iteration sequence. 
A generalization of the theorem to d-minimal expansions $\WR$ of $\R_{an}$ fails. On the other hand, we prove our theorem for expansions $\WR$ of arbitrary real closed fields. Moreover, its conclusion holds for certain structures with d-minimal open core, such as $\la \overline \R, \R_{alg}, 2^\Z\ra$.



\end{abstract}
 \maketitle

\section{Introduction}

In this paper we provide conditions under which an expansion $\WR$ of the real field $\overline \R$ by a discrete set introduces no new smooth definable functions. This property can be viewed as a tameness condition for $\WR$, much alike the property of introducing no new open definable sets for expansions of $\overline \R$ by a dense-codense set (\cite{bh}). Note that neither of these properties is captured by the usual notions of tameness, since both $(\overline \R, exp)$ and $(\overline \R,\Z)$ define new smooth functions and open sets, yet they are at the two extremes of the model-theoretic complexity hierarchy, the first structure being o-minimal, the second defining the whole projective hierarchy.

The targeted example of the current work is the expansion $\WR=\la \overline \R, P\ra$ of the real field by $P=2^\Z$, first studied by L. van den Dries \cite{VDD1}.
It was proven there that every definable set is given by an existential formula, where the quantifier ranges over $P$ and the rest of the formula is semialgebraic. As a consequence, every definable subset of $\R$ is the union of an open set and finitely many discrete sets. Subsequently, C. Miller \cite{Mil1} extended this study to the context of \emph{d-minimal} structures, which are defined exactly as the expansions of $\la \R, <\ra$ that satisfy the last property, uniformly in parameters. By now, many structures are known to be d-minimal, such as expansions of the real field by an iteration sequence \cite{MT}, a fast sequence \cite{Mil2} or a spiral \cite{Mil1}.

In this paper, we take a different approach and set as our starting point a decomposition property for definable sets, which we introduce next. Although our main interest is in expansions of the real field, we work here in the general setting of expansions $\WR$ of an arbitrary real closed field $\cal R$.\\

\textbf{From the rest of this paper, we fix an expansion
$\WR=\la \cal R, \dots\ra$ of a real closed field $\cal R=\la R, <, +, \cdot\ra$. By `definable' (respectively, `semialgebraic'), we mean definable in $\WR$ (respectively, in \cal R), with parameters.}\\

The real closed field $\cal R$ is an example of an o-minimal structure. We assume familiarity with the basics of o-minimality, in particular with the o-minimal topology and the notions of a cell and cell decomposition. We call a semialgebraic set \emph{connected} if it is not the union of two proper relatively open semialgebraic subsets. Equivalently, any two points in it can be connected via a semialgebraic path.

For \emph{any} set $X\sub R^n$, we define its \emph{dimension} as the maximum $k$ such that some projection of $X$ to $k$ coordinates has non-empty interior.


\begin{defn}
Let $Y\sub X\sub R^n$ be two sets. We say that $Y$ is a \emph{semialgebraic chunk of $X$} if it is a semialgebraic cell, $\dim Y= \dim X$, and for every $y\in Y$, there is an open box $B\sub R^n$ containing $y$ such that $B\cap X\sub Y$. 
\end{defn}

For example, if $X= 2^\Z\cup\{0\}$, then every singleton contained in $X$  is a semialgebraic chunk, except for $\{0\}$.

\begin{defn} We say that $\WR$ has the \emph{decomposition property} (DP) if for every definable set $X\sub \R^n$,
\begin{enumerate}
  \item[(I)] there is a semialgebraic family $\{Y_t\}_{t\in R^m}$ of subsets of $R^n$, and a definable set $S\sub R^m$ with $\dim S=0$, such that $X=\bigcup_{t\in S} Y_t$,

  \item[(II)] $X$ contains a semialgebraic chunk.
\end{enumerate}
\end{defn}

\begin{defn}
We say that $\WR$ has the \emph{dimension property} (DIM) if for every definable family $\{X_t\}_{t\in S}$ of semialgebraic sets, with $\dim S=0$, we have
$$\dim \bigcup_{t\in S} X_t=max_{t\in S} \dim X_t.$$
\end{defn}

Our main theorem is the following (Theorem \ref{k4}).

\begin{theorem}\label{main}
Let $\WR$ be an expansion of a real closed field $\cal R$, satisfying \textup{(DP)} and \textup{(DIM)}. Let $f: X\sub R^n\to R$ be a definable smooth function, with an open semialgebraic domain $X$. Then $f$ is semialgebraic. 
\end{theorem}

\begin{rmk}\label{rmk-thm} $ $
\begin{enumerate}
              \item If $R=\R$ and $\WR$ is d-minimal, then (DIM)  becomes redundant. Indeed, in this case, a definable set is $0$-dimensional if and only if it is discrete if and only if it is countable. Hence (DIM) holds by the Baire Category Theorem, see Fact \ref{dim} below.

              \item If $R=\R$ and (DP)(I) holds with $S$ countable instead of $\dim S=0$, then every definable set $X$ is countable if and only if it is $0$-dimensional.  Indeed, if $\dim X=0$, then by the adapted (DP)(I), it is a countable union of finite sets, and hence countable. If $\dim X>0$, then by (DP)(II), it must contain an infinite semialgebraic set, and hence it is uncountable. In particular, (DIM) in this case again becomesredundant, by Fact \ref{dim}, and we also obtain the ``countable version" of the theorem stated in the abstract (where by a ``uniform countable decomposition", we mean a decomposition as in (DP)(I) with $S$ countable.)




                 \item If we replace ``smooth" by ``analytic" in Theorem \ref{main}, then the conclusion becomes clear, since an analytic function which is semialgebraic on an open subset of its domain is semialgebraic. On the other hand, if $\WR$ is d-minimal and not o-minimal, then it is possible to define new $\cal C^n$ functions, for every $n\in \N$, see Remark \ref{rmk-optimal}(3) below.


            \end{enumerate}
\end{rmk}



In Section \ref{sec-examples}, we apply our theorem to various examples.

\begin{prop}\label{prop-examples}
Let $\WR$ be any of the following structures:
\begin{enumerate}
  \item   $\la \overline \R, \alpha^\Z\ra$, where $1<\alpha\in \R$, 
  \item   $\la \overline \R, P\ra$, where $P$ is an iteration sequence,
\end{enumerate}
Then $\WR$ satisfies \textup{(DP)}. Hence (by Remark \ref{rmk-thm}(1)), the conclusion of Theorem \ref{main} holds.
\end{prop}

We suspect that (DP) holds in more d-minimal expansions of the real field, such as expansions by fast sequences (Question \ref{qn-fast}). Here we apply our theorem to structures that go even beyond the d-minimal setting, and where (DP)(II) actually fails (Corollary \ref{cor-khani}). 

 \begin{cor}
 Let $\WR=\la \overline \R, \R_{alg}, 2^\Z\ra$, where $\R_{alg}$ is the field of real algebraic numbers. Then every definable smooth function with open semialgebraic domain is semialgebraic.
 \end{cor}

Finally, by Example \ref{exa-Philipp}, due to P. Hieronymi, if we replace $\cal R$ by an o-minimal expansion $\cal R'$ of a real closed field (such as $\R_{an}$), and  `semialgebraic' by `definable in $\cal R'$', then Theorem \ref{main} fails.

$ $\\\noindent\textbf{Strategy of the proof.}
In the rest of this introduction, we describe the strategy of our proof of Theorem \ref{main}, and provide a basic example. The proof uses a key fact from real algebraic geometry  (\cite[Proposition 8.4.1]{BCR} or Fact \ref{Coste} below), that the Zariski closure of a connected Nash-submanifold is irreducible. This allows us to establish that under certain conditions, the Zariski closures of certain graphs of two functions coincide (Corollary \ref{k3}).

\begin{example}
Let $\WR=\la \overline \R, 2^\Z\ra$ and $f: \R^{>0}\to \R$ be a definable smooth function which is semialgebraic on each $(2^n, 2^{n+1})$, $n\in \Z$. Denote $f_n= f_{\res (2^n, 2^{n+1})}$. By induction on $|n|$, and Lemma \ref{L-def}, it suffices to prove that the graphs $\Gamma(f_n)$ and $\Gamma(f_{n+1})$ are contained in the same irreducible algebraic set of dimension $1$ (namely, that their Zariski closures $\overline{\Gamma(f_n)}^{zar}$ and $\overline{\Gamma(f_{n+1})}^{zar}$ coincide). 
If $f_n$ and $f_{n+1}$ are given by polynomials, then it is not hard to show that the smoothness condition on $f$ forces (the degrees and coefficients of) those polynomials to be the same.
In general, the graphs $\Gamma(f_n)$ and $\Gamma(f_{n+1})$ are contained in some $1$-dimensional algebraic sets, and the smoothness condition on $f$ forces the ``concatenation'' of their graphs to be a connected Nash-submanifold (Definition \ref{def-Nashsub}). By Fact \ref{Coste}, its Zariski closure is irreducible. This implies
$\overline{\Gamma(f_n)}^{zar}=\overline{\Gamma(f_{n+1})}^{zar}$(Corollary \ref{k3}), as needed.
\end{example}

For arbitrary definable functions in $\la \overline \R, 2^\Z\ra$, the above inductive proof cannot work because, even though by (DP)(I) the graph of $f$ is a countable union of graphs of semialgebraic maps $f_t:Z_t\sub \R^n\to \R$, $t\in S$, it need not be true that every $f_t$ can be ``concatenated'' with some other $f_{t'}$ and yield a connected Nash-submanifold. For example, suppose that the above $f$ is extended to the whole $\R$, via some semialgebraic map $g:(-\infty, 0)\to \R$. Then there is no $n\in \Z$, such that the set $[-\infty, 0] \cup [2^n, 2^{n+1}]$ is connected.

We are thus led to prove Theorem \ref{main} by contradiction, as follows. Using (DP)(I) and Lemma \ref{decomp_pol}, we can find a $0$-dimensional definable family of algebraic sets  $\{Y_t : t\in S\}$ in $\R^n$ , such that  for every $t\in S$, the graph of $f_t$ is contained in $Y_t$. Using a basic dimension-theoretic argument, it suffices to prove that for any $t, t'\in S$ with $\dim Y_t=\dim Y_{t'}=n$, we have $\overline{\Gamma(f_t)}^{zar}=\overline{\Gamma(f_{t'})}^{zar}$. Assume not. We apply Proposition \ref{a1}, which is proved using (DP)(II), to find distinct $f_t$ and $f_{t'}$, which are Nash-concatenable. Thus, again by Corollary \ref{k3}, we obtain $\overline{\Gamma(f_t)}^{zar}=\overline{\Gamma(f_{t'})}^{zar}$, a contradiction.

$ $\\
\noindent\emph{Structure of the paper.}
In Section \ref{sec-prel}, we fix notation and collect some basic facts. In Section \ref{sec-proof}, we prove Theorem \ref{main}, after some preparatory work and an interlude of real algebraic geometry. In Section \ref{sec-examples}, we apply our theorem to certain examples and state some open questions.

$ $\\
\noindent\textbf{Acknowledgments.}
We thank Chris Miller and Patrick Speissegger for their important feedback, Claus Scheiderer and Markus Schweighofer for helpful discussions over real algebraic geometry, and Philipp Hieronymi and Erik Walsberg for pointing out further examples and asking relevant questions.




\section{Preliminaries}\label{sec-prel}



We assume familiarity with the basics of o-minimality, as they can be found for example in \cite{VDD}. 
By a $k$-cell, we mean a semialgebraic cell of dimension $k$.  If $S\sub R^n$ is a set, its closure, interior and boundary are denoted by $\cl S, \intr(S), \bd(S):=\cl S\sm \intr(S)$, respectively. Sole exception to this notation is that of the real field $\overline \R$.
We call a family $\{X_t\}_{t\in S}$ of sets $X_t\sub R^n$ semialgebraic (respectively, definable) if the set $\bigcup_{t\in S}\{t\}\times X_t$ is semialgebraic (respectively, definable). By an open box $B\sub R^n$, we mean a set of the form
$$B=(a_1, b_1)\times \dots \times (a_n\times b_n),$$
for some $a_i< b_i\in R\cup\{\pm\infty\}$. By an open set we always mean a non-empty open set.  Unless stated otherwise, $\pi: R^{n}\to R^{n-1}$ denotes the coordinate projection onto the first $n-1$ coordinates. By a `smooth' function on an open domain we mean a function of class $\cal C^\infty$; namely, it is $n$-differentiable for every $n\in \N$.

 We prove three basic lemmas that will be used in the sequel.

\begin{lemma}\label{hop}
Let $X\subseteq R^n$ be an open semialgebraic set and $Y\subseteq \overline X$ a definable set, such that $\dim(X\setminus Y)<n$. Then  $\cl{Y}=\cl{X}$.
\end{lemma}
\begin{proof}
Let $x\in \cl{X}$. Let $B\sub R^n$ be any open box containing $x$. Since $X$ is open, there is an open box $B'\sub B\cap X$. Since $\dim (X\sm Y)<n$, $B'\cap Y\ne \es$, and hence $B\cap Y\ne \es$. Thus $x\in \overline Y$, as needed.
\end{proof}
 \begin{lemma}\label{var} Assume \textup{(DP)(II)}. Let $Y\sub R^n$ be an open definable set. Then $\dim \bd(Y)< n$.
\end{lemma}
\begin{proof}
Let us assume for contradiction that $\dim \bd(Y)=n$. By (DP)(II), $\bd(Y)$ contains an open box $B$. Since $B\subseteq \bd(Y)$, we have $B\cap Y=\emptyset$. But then $X\cap \overline Y=\es$, a contradiction.
\end{proof}

\begin{lemma}\label{l1}
Let $A_1,\,A_2 \subseteq R^n$ be two  open, connected and disjoint  semialgebraic sets. Suppose there is an open box $B\subseteq \intr(\cl{A_1}\cup \cl{A_2})$, such that $B\cap A_i\neq \emptyset$, for $i=1,2$. Then $\intr(\cl{A_1}\cup \cl{A_2})$ is connected.
\end{lemma}
\begin{proof}
Denote $X= \intr(\cl{A_1}\cup \cl{A_2})$, and take any $x, y\in X$. We need to find a semialgebraic path in  $X$ connecting $x$ and $y$. Take an open box  $B_1\sub X$  containing $x$. Clearly, $B_1$ must intersect $A_i$, for some $i=1, 2$, and since $B_1$ is connected, we may assume that $x\in A_i$. Similarly, we may assume $y\in A_i$, for some $i=1,2$. By assumption, there is an open box $B\sub X$ with each $B\cap A_i\ne\es$. Since all of  $B$, $A_1$ and $A_2$ are connected and contained in $X$, the result easily follows.
\end{proof}

\section{Proof of Theorem \ref{main}}\label{sec-proof}

The proof takes place in Subsection \ref{subsec-proof} below. As mentioned in the introduction, a crucial step of our strategy is to show that certain functions $f_i: X_i\to R$ are ``Nash-concatenable". This notion is introduced in Subsection \ref{sec-rag}. As a first step, one needs to establish that $\intr(\overline X_1 \cup \overline X_2)$ is connected. Towards that goal, we prove Proposition \ref{a1} below.

\subsection{Towards concatenation}

A simple instance of Proposition \ref{a1} is the following. Suppose that $\{W_t: t\in S\}$ is a definable family of disjoint open intervals in $R$, such that $R=\overline{\bigcup_{t\in S} W_t}$. Then there are $t\ne t'\in S$ such that $\intr(\cl{W_t}\cup \cl{W_{t'}})$ is connected.

\begin{proposition}\label{a1} Assume \textup{(DP)(II)}.
Let $\{W_t : t\in S\}$ be a definable family of open sets in $R^n$,  such that:
\begin{enumerate}
 \item[(i)] for every $t\in S$, $W_t=\intr (\cl{W_t})$,
  \item[(ii)] there is  an open cell $X\sub R^n$, such that $\overline X = \overline{\bigcup_{t\in S} W_t}$,
  \item[(iii)] for every $t, t'\in S$, $W_t\ne W_{t'}\Rarr\, W_t\cap W_{t'}=\es$,
\item[(iv)] there are $t, t'\in S$, $W_t\ne W_{t'}$.
\end{enumerate}
Then there are an open box $B\sub X$ and $W_{t_1}\ne W_{t_2}$, such that
\begin{enumerate}
\item  $B\cap W_{t_i}$ is semialgebraic, open and connected, $i=1,2$,

 \item  $\intr (\cl{B\cap W_{t_1}}\cup \cl{B\cap W_{t_2}})$ is connected.
\end{enumerate}
\end{proposition}

\begin{proof}


By (iii) and since each $W_t$ is open, we have $\bigcup_{t\in S} \bd(W_t)\subseteq \bd(\bigcup_{t\in S} W_t)$. By Lemma  \ref{var}, $\dim \bigcup_{t\in S} \bd(W_t)\le n-1$.\\


\noindent\textbf{Claim.} \emph{$\dim  \bigcup_{t\in S} \bd(W_{t}) \cap X=n-1$.}
\begin{proof}[Proof of Claim] It suffices to find $t\in S$, such that $\dim \bd(W_{t}) \cap X=n-1$. By induction on $n$. For $n=1$, $X$ is an open interval and the statement follows easily from (iv).  
Let $n>0$. We split into two cases.\smallskip

\noindent\textbf{Case I:} For all $t\in S$, $\pi^{-1}(\pi(W_t))\cap X\subseteq W_t$. It is then not hard to see that for all $t\in S$ and $x\in \bd(\pi(W_t))$,
$$\pi^{-1}(x)\cap X\subseteq \bd(W_t)\cap X,$$
and   the family $\{\pi(W_t) :  t\in S\}$ satisfies Conditions (i)-(iv). By inductive hypothesis, there is
 $t\in S$, such that $\bd(\pi(W_t))\cap \pi(X)$ has dimension $n-2$. Thus, since $X$ is an open cell, the set $$\pi^{-1}(\bd(\pi(W_t))\cap \pi(X))\cap X\subseteq \bd(W_t)\cap X$$ has dimension $n-1$.\smallskip

\noindent\textbf{Case II:}  Otherwise, there are $t_1, t_2\in S$ with $W_{t_1}\ne W_{t_2}$ and $x_i\in W_{t_i}$, such that $\pi(x_1)=\pi(x_2)$. Since each $W_{t_i}$ is open, the set $A\subseteq \pi(W_{t_1})\cap \pi(W_{t_2})\sub R^{n-1}$ is open. It suffices to show that for all $x\in A$, there is $y\in \bd(W_{t_1})\cap X$, such that $\pi(y)=x$. Since $X$ is an open cell, $\pi^{-1}(x)\cap X$ is connected.  The sets $W_{t_i}$ being disjoint by (iii), and each $\pi^{-1}(x)\cap W_{t_i}\neq \emptyset$ open, we obtain the result.
\end{proof}

By the claim and (DP)(II), there is  a semialgebraic set $Y\sub \bigcup_{t\in S}\bd(W_{t})\cap X$ of dimension $n -1$, and an open box $B\sub R^n$, such that
$$\emptyset\neq B\cap \bigcup_{t\in S}\bd(W_{t})\cap X\subseteq Y.$$
Since $X$ is open, we may assume $B\sub X$. Moreover, since $Y$ is semialgebraic and has dimension $n-1$, by taking $B$ small enough, we may assume that  $B\setminus Y$ has two open semialgebraic connected components, $B_1$ and $B_2$. By (ii), Lemma \ref{var}, and since each $B_i\sub X$, we can find $t_1, t_2\in S$, such that $B_i\cap W_{t_i}\ne \es$, $i=1,2$.

We next show that $W_{t_1}\ne W_{t_2}$.
First observe that since $B\cap \bigcup_{t\in S}\bd(W_t)\subseteq Y$ and $B_i\cap Y=\emptyset$, we have  $B_i\cap \bigcup_{t\in S} \bd(W_t)=\emptyset$. In particular, $B_i\cap \bd(W_{t_i})=\emptyset$. Therefore, $B_i\subseteq W_{t_i}$, $i=1,2$. Since $B$ is an open box and $\dim (B\sm (B_1\cup B_2))<n$, we have
$$(*)\,\,\,\,\,B\sub \intr( \cl{B_1}\cup \cl{B_2})\sub \intr(\cl{W_{t_1}}\cup \cl{W_{t_2}}).$$
If $W_{t_1}=W_{t_2}$, then by (i), we obtain
$$B\subseteq \intr(\cl{B_1}\cup \cl{B_2})=\intr(\cl{W_{t_1}}\cup \cl{W_{t_2}})= \intr(\cl{W_{t_1}})=W_{t_1},$$
contradicting $B\cap \bigcup_{t\in S}\bd(W_t)\neq \emptyset$, (iii) \&(iv).

Now we are ready to show (1). It suffices to prove that $B\cap W_{t_i}=B_i$, for $i=1,2$. As shown earlier, we have $B_i\subseteq W_{t_i}\cap B$. On the other hand, observe that
$$B\cap W_{t_1}=(B_1 \cap W_{t_1}) \cup (B_2\cap W_{t_1}) \cup (B\cap Y\cap W_{t_1}).$$
Since $B_2\sub W_{t_2}$ and $W_{t_1}\cap W_{t_2}=\es$, the second part of the above union is empty. Moreover, since $Y\sub \bigcup_{t\in S} \bd(W_t)$, we have $Y\cap W_{t_1}=\es$. Therefore, the third part of the above union is also empty. Hence $B\cap W_{t_1}=B_1$. Similarly, one shows $B\cap W_{t_2}=B_2$, as needed.

Finally, (2) follows from (1), $(*)$ and Lemma \ref{l1}, for $A_i=B\cap W_{t_i}$.
\end{proof}

\subsection{Real algebraic geometry}\label{sec-rag}

In this section we recall  some basic facts from real algebraic geometry. A standard reference is Bochnak-Coste-Roy \cite{BCR}. The key fact is Fact \ref{Coste} below which states that the  Zariski closure of a connected Nash-submanifold is irreducible. This allows us to establish that under certain conditions, the Zariski closures of the graphs of two Nash-concatenable functions coincide (Corollary \ref{k3}).

\begin{definition}\label{def-nash}
By an \emph{algebraic set $A\sub R^n$}, we mean the zero set $P^{-1}(0)$ of a polynomial $P\in R[X]$.  The \emph{Zariski closure} of a set $V\sub R^n$ is the intersection of every algebraic set containing $V$,  denoted by  $\overline V^{zar}$. Note that $\overline V^{zar}$ is algebraic, because $R[X_1,\ldots, X_n]$ is Noetherian.

 Let $V$ be an algebraic set. We say that $V$ is \emph{irreducible} if, whenever $V=V_1\cup V_2$, with each $V_i$ algebraic, we have $V=V_i$, for $i=1$ or $2$. The \emph{Krull-dimension} of $V$, denoted by $\dim_K(V)$,  is the maximum $n$ such that there is a sequence of irreducible algebraic sets $V_0\subsetneq \ldots \subsetneq V_n\subseteq V$.
\end{definition}

\begin{fact}\label{Kdim}
Let $X$ be a semialgebraic set. Then
$$\dim(X)=\dim_K(X)=\dim({\overline X}^{zar}).$$
\end{fact}
\begin{proof}
For the first equality: by \cite[Proposition 2.8.4]{BCR}, if $A\sub R^n$ is an open semialgebraic set, then $\dim_K(A)=n$. By \cite[Proposition 2.8.7]{BCR}, the graph of a semialgebraic function  has Krull-dimension equal to the Krull-dimension of its domain. By \cite[Proposition 2.8.5]{BCR}, the Krull-dimension of a finite union of semialgebraic sets equals the sum of their Krull-dimensions. Now, let $X$ be a semialgebraic set. By cell decomposition, $X$ is a finite union of graphs of semialgebraic functions with open domains, and $\dim X$ equals the maximum of their dimensions. Combining the above three propositions, the first equality follows.

The second equality is by \cite[Proposition 2.8.2]{BCR}.
\end{proof}

\begin{lemma}\label{I1}
Let $U\subseteq V\subseteq R^n$ be two algebraic sets of the same dimension. If $V$ is irreducible, then $U=V$.
\end{lemma}
\begin{proof}
By Fact \ref{Kdim}, we have $\dim_K(U)=\dim_K(V)$, and hence the result follows from the definition of Krull-dimension.
\end{proof}

\begin{defn}(\cite[Definitions 2.9.3, 2.9.9]{BCR})\label{def-Nashsub}
A \emph{Nash function} $f: X\sub R^n\to R^m$ is a semialgebraic smooth function with open domain. A \emph{Nash-diffeomorphism}  $f:X\to Y$ is a Nash function which is a bijection and whose inverse is also Nash.

A semialgebraic set $V\subseteq R^m$ is a \emph{Nash-submanifold of dimension $d$} if,  for every $x\in V$, there is a Nash-diffeomorphism $\phi$ from an open semialgebraic neighborhood $U$ of the origin in $R^m$ onto an open semialgebraic neighborhood  $U'$ of $x$ in $R^m$, such that $\phi(0)=x$ and $\phi((R^d\times \{0\})\cap U)=V\cap U'$
\end{defn}

Note  that the graph of a Nash function with connected domain is a connected Nash-submanifold.

We introduce the following notion.

\begin{defn}
Let $f:X\sub R^n\rightarrow R^m$, $g:Y\sub R^n\rightarrow R^m$ be two semialgebraic maps with open connected domains. We say that \emph{$f$ and $g$ are Nash-concatenable}, if the set $A=\intr(\overline X\cup \overline Y)$ is connected, and there is a (unique) Nash function $h: A\to R^m$ that extends both $f$ and $g$. We denote $h=f^\smallfrown g$.
\end{defn}

The uniqueness of the above map $h$ follows from its continuity and the fact that $X\cup Y$ is dense in its domain $A$. Note also that since $A$ is connected, $\Gamma(h)$ is a connected Nash-submanifold.

The key fact from \cite{BCR} is the following.

\begin{fact}\label{Coste}
Let $V\subseteq R^m$ be a connected  Nash-submanifold. Then $\overline V^{zar}$ is irreducible.
\end{fact}
\begin{proof}
See \cite[Proposition 8.4.1]{BCR}.
\end{proof}

\begin{corollary}\label{k3}
Let $f$ and $g$ be two Nash-concatenable functions. Then $$\cl{\Gamma(f)}^{zar}=\cl{\Gamma(g)}^{zar}.$$
\end{corollary}
\begin{proof}
Since $\Gamma (f\smallfrown g)$ is a connected Nash-submanifold, by Fact \ref{Coste}, $\cl{\Gamma (f\smallfrown g)}^{zar}$ is irreducible. Note that $\dim \Gamma(f)=\dim \Gamma(f\smallfrown g)=\dim \Gamma(g)$, and hence, by Lemma \ref{Kdim},
$$\dim \cl{\Gamma(f)}^{zar}=\dim \cl{\Gamma(f\smallfrown g)}^{zar}=\dim \cl{\Gamma(g)}^{zar} .$$
By Lemma \ref{I1}, $\cl{\Gamma(f)}^{zar}=\cl{\Gamma(f\smallfrown g)}^{zar}=\cl{\Gamma(g)}^{zar}$.
\end{proof}

\begin{cor}\label{Zaropen} Let $f:X\sub R^n\to R$ be a Nash function with connected domain, and let $U\sub X$ be open. Then $$\cl{\Gamma({f})}^{zar}=\cl{\Gamma(f_{\res U})}^{zar}.$$
\end{cor}
\begin{proof}
By Corollary \ref{k3}, for $g=f_{\res U}$.
\end{proof}

\subsection{Proof of Theorem \ref{main}}\label{subsec-proof}

We will need two preliminary lemmas. If \cal C is a collection of subsets in $R^{n+1}$, by $\pi(\cal C)$ we denote the collection of their projections on $R^n$.

\begin{lemma}\label{L-def}
Let $f:X\subseteq R^n\rightarrow R$ be a continuous map, such that $X$ is a
$k$-cell and $\Gamma(f)\subseteq Y$, with $Y$ a semialgebraic set of
dimension $k$. Then $f$ is semialgebraic.
\end{lemma}
\begin{proof}
We may clearly assume that $\pi(Y)=X$. Let $\cal D$ be a cell decomposition of $R^{n+1}$ partitioning $Y$, and $\cal C\sub \cal D$  the collection of those cells that intersect $Y$. So $\pi(\cal C)$ is a finite partition of $X$. It suffices
to show that for every $S\in \pi(\cal C)$, $f_{\res S}$ is semialgebraic. Fix
such $S$. By definition of cell decomposition, we have
$$Y\cap (S\times R)= Y_1\cup\dots\cup Y_m,$$
for some disjoint $k$-cells $Y_i\sub R^{n+1}$,  with $\pi(Y_i)=S$. Of course,
since $\Gamma(f)\sub Y$, we have $\Gamma(f_{\res S})\sub Y\cap (S\times R)$.
Since the $Y_i$'s are disjoint and $f_{\res S}$ is continuous, $\Gamma(f_{\res
S})$ must equal one of the $Y_i$'s. Hence it is semialgebraic.
\end{proof}

\begin{lemma}\label{decomp_pol}
Let $X\sub R^{m+l}$ be a semialgebraic set and $A$ its projection onto the last $l$ coordinates.
Then there are finitely many polynomials $P_i\in R[X,T]$ and semialgebraic sets $A_i$, $i=1, \dots, s$, such that $A=\bigcup_i A_i$, and for every $i$ and $t\in A_i$,
$$(*) \,\,\,\,\,\, \text{ $X_t\sub P_i(-,t)^{-1}(0)$\,\,\,  and \,\,\,\,$\dim X_t = \dim P_i(-,t)^{-1}(0)$.}$$
\end{lemma}
\begin{proof}
By partitioning $A$ into finitely many sets, we may assume that there is $n$ with each $\dim X_t=n$. We perform induction on $k=\dim A$. For $k=0$, the result follows from Fact \ref{Kdim}, applied to each $X_t$.  Let $k>0$. By Fact \ref{Kdim} again, there is $P_0\in R[X,T]$ such that $X\subseteq P_0^{-1}(0)$ and both sets have dimension $n+k$. Let
$$A_0=\{t\in A : \dim P_0(-,t)^{-1}(0)>n \}.$$
Then clearly, for every $t\in A\sm A_0$, $(*)$ holds, with $i=0$. Moreover, by o-minimality, $\dim A_0<k$. Hence, by inductive hypothesis, the statement of the lemma also holds for $X\cap (R^m\times A_0)$, and hence it holds for $X$.
\end{proof}

We are now ready to prove our main theorem.

\begin{theorem}\label{k4} Assume \textup{(DP)} and \textup{(DIM)}. Let $f:X\subseteq R^n\rightarrow R$ be a definable smooth function, where $X$ is an open semialgebraic  set. Then $f$ is semialgebraic.
\end{theorem}
\begin{proof} By cell decomposition, we may assume that $X$ is an open cell.
By (DP)(I), there is a semialgebraic family $\{Y_t\}_{t\in R^m}$ of sets $Y_t\sub R^{n+1}$, and a definable $S\sub R^m$ with $\dim S=0$, such that $\Gamma(f)=\bigcup_{t\in S} Y_t$. By cell decomposition again, we may assume that each $Y_t$ is a cell. For every $t\in S$, write $Z_t=\pi(Y_t)$ and $f_t = f_{\res Z_t}$. So $Y_t=\Gamma(f_t)$. Observe that $\dim Z_t=n$ if and only if $Z_t$ is an open cell. Let
$$A=\{ t\in S : \dim Z_t = n\}.$$
Since $X=\bigcup_{t\in S} Z_t$, and for every $t\in S\sm A$, we have $\dim Z_t<n$, by (DIM), we obtain
\begin{equation}
\dim \left(X\sm \bigcup_{t\in A} Z_t\right)<n.\label{Zt}
\end{equation}
Since $X$ is an open cell, by Lemma \ref{hop}, we have $X\sub \overline{\bigcup_{t\in A} Z_t}$. By continuity of $f$,
\begin{equation}
\Gamma(f)\sub \overline{\bigcup_{t\in A} \Gamma(f_t)}.\label{ft}
\end{equation}

By Lemma \ref{decomp_pol}, there are finitely many polynomials $P_i\in R[X,T]$ and semialgebraic sets $C_i$, $i=1, \dots, s$, such that $R^n=\bigcup_i C_i$, and for every $i$ and $t\in R_i$,
$$ \text{ $Y_t\sub P_i(-,t)^{-1}(0)$  and $\dim P_i(-,t)^{-1}(0)=\dim Y_t$.}$$
In particular, for every $i$ and $t\in C_i\cap A$,
$$ \text{ $\Gamma(f_t)=Y_t\sub P_i(-,t)^{-1}(0)$  and $\dim P_i(-,t)^{-1}(0)\le n$.}$$
For every $i$, denote $A_i=C_i\cap A$. So $A=\bigcup_i A_i$. For every $t\in A$, define
$$B_t=\bigcup_i\{t'\in A_i : \Gamma(f_t)\sub P_i(-,t')^{-1}(0) \}\sub A.$$
Observe that for every $t\in A$, we have $t\in B_t$, and hence $\bigcup_{t\in A} B_t=A$.

Our goal is to show that for every $t, t'\in A$, $B_t=B_{t'}$. This will imply the conclusion of the theorem. Indeed, in this case, we obtain that for every $t\in A$, $B_t=A$. Now,  fix $a\in A$, say $a\in A_i$. Take any $t\in A$. So $a\in B_t$. That is, $\Gamma(f_t)\sub P_i(-,a)^{-1}(0)$. Together with (\ref{ft}), this implies that
$$\Gamma(f)\sub \overline{\bigcup_i P_i(-,a)^{-1}(0)},$$
and since the  set on the right is semialgebraic and has dimension $n$, the conclusion of the theorem follows from Lemma \ref{L-def}.


We achieve our goal in Claim 3 below. First, we need two preliminary claims.

$ $\\
\noindent\textbf{Claim 1.} \emph{Let $t,t'\in A$. Then} $$Z_t\cap Z_{t'}\ne\es\,\,\Rarr\,\, \cl{\Gamma(f_t)}^{zar}=\cl{\Gamma(f_{t'})}^{zar}\,\,\Rarr\,\, B_t=B_{t'}.$$
 \begin{proof}[Proof of Claim 1]
For the first implication, let $U=Z_t\cap Z_{t'}$. Since $U$ is open, by Corollary \ref{Zaropen}, $\cl{\Gamma(f_t)}^{zar}=\cl{\Gamma(f_{\res U})}^{zar}=\cl{\Gamma(f_{t'})}^{zar}$.

For the second implication, let $x\in B_t$. Hence $\Gamma(f_t)\subseteq P_i(-,x)^{-1}(0)$. By assumption  $\Gamma(f_{t'})\subseteq P_i(-,x)^{-1}(0)$. Hence $x\in B_{t'}$.
 \end{proof}

Let us define the equivalence relation $\sim$ on $A$, by $t\sim t'$ if and only if $B_t=B_{t'}$.
For $t\in A$, let
$$W_t=\intr\left(\cl{\bigcup_{t\sim t'\in A} Z_{t'}}\right).$$
Clearly, $\{W_t:t\in A\}$ is a definable family of open sets in $R^n$.\\

\noindent\textbf{Claim 2.} \emph{We have:
\begin{enumerate}
 \item[(i)] for every $t\in A$, $W_t=\intr (\cl{W_t})$,
  \item[(ii)] $\overline X = \overline{\bigcup_{t\in T_i} W_t}$,
  \item[(iii)] for every $t, t'\in A$,
$$W_t\ne W_{t'} \,\Rarr\, t\not\sim t'\in A\,\Rarr\, W_t\cap W_{t'}=\es.$$
\end{enumerate}}
\begin{proof}[Proof of Claim 2]
(i) Clear, by definition of $W_t$.

(ii) For every $t\in A$, since $X$  contains $\bigcup_{t\sim t'\in A} Z_{t'}$, we have that $W_t\sub \overline X$. Moreover, since $Z_t$ is open, by definition of $W_t$ we have $Z_t\sub W_t$. Hence, by (\ref{Zt}), $$\dim \left(X\sm \bigcup_{t\in A} W_t\right)<n,$$ and by Lemma \ref{hop}, $\overline X = \overline{\bigcup_{t\in A} W_t}$.

(iii) The first implication is clear. So we need to prove that if $t\not\sim t'\in A$, then $W_t\cap W_{t'}=\es$. Assume otherwise. Note that $W_t\cap W_{t'}$ is open. Hence the intersection
$$\cl{\bigcup_{t\sim s\in A} Z_{s}}\cap \cl{\bigcup_{t'\sim s'\in A} Z_{s'}}$$
contains an open set. But then,  by (DIM), there must be $s\not\sim s'$ with $Z_s\cap Z_{s'}\ne \es$, contradicting Claim 1.
\end{proof}

 In particular, the family $\{W_t : t\in A\}$ satisfies properties Lemma \ref{a1}(i)-(iii).

$ $\\ \noindent\textbf{Claim 3.} {\em For every $t,t'\in A$, $B_t=B_{t'}$.}
\begin{proof}[Proof of Claim 3] Assume not. By Claim 2(iii), the family $\{W_t : t\in A\}$ satisfies all properties of Lemma \ref{a1}. By that lemma, there are an open box $B\sub R^n$ and $t_1, t_2\in A$ with $W_{t_1}\ne W_{t_2}$, such that each $B\cap W_{t_j}$ is semialgebraic, open and connected, and
 $\intr (\cl{B\cap W_{t_1}}\cup \cl{B\cap W_{t_2}})$ is connected.
That is, $f_{\upharpoonright B\cap W_{t_1}}$ and $f_{\upharpoonright B\cap W_{t_2}}$ are Nash-concatenable. By Corollary \ref{k3}, $\cl{\Gamma(f_{\upharpoonright B\cap W_{t_1}})}^{zar}=\cl{\Gamma(f_{\upharpoonright B\cap W_{t_2}})}^{zar}$. Since each $B\cap W_{t_j}$ is open, by (DIM), there are $a_1, a_2\in A$ with $a_j\sim t_j$, such that $B\cap Z_{a_j}$ is open, $j=1, 2$. Since $B\cap W_{t_j}$ is connected, open and semialgebraic, by Lemma \ref{Zaropen}, we obtain
$\cl{\Gamma(f_{\upharpoonright B\cap W_{t_j}})}^{zar}=\cl{\Gamma(f_{\upharpoonright B\cap Z_{a_j}})}^{zar}$, $j=1,2$,
and hence $$\cl{\Gamma(f_{\upharpoonright B\cap Z_{a_1}})}^{zar}=\cl{\Gamma(f_{\upharpoonright B\cap Z_{a_2}})}^{zar}.$$ By Claim 1, $t_1\sim a_1\sim a_2\sim t_2$, contradicting Claim 2(iii).
\end{proof}
This completes the proof of the theorem.
\end{proof}


\section{Examples}\label{sec-examples}

In this section we establish Conditions (DP) and (DIM) in the examples of Proposition \ref{prop-examples}, and state some open questions. In those examples, $R=\R$ and $\WR$ is d-minimal.  Hence, as stated in Remark \ref{rmk-thm}(1), Condition (DIM) becomes redundant, based on the following fact.






 \begin{fact}\label{dim} For every countable family $\{X_t\}_{t\in S}$ of semialgebraic sets $X_t\sub R^n$, we have
$$\dim \bigcup_{t\in S} X_t=max_{t\in S} \dim X_t.$$
\end{fact}
\begin{proof}
Towards a contradiction, suppose $\dim \bigcup_{t\in S} X_t=m>l=max_{t\in S} \dim X_t.$ After projecting onto some $m$ coordinates, we may assume that $m=n$. Hence, $\bigcup_{t\in S} X_t$ has non-empty interior. By the Baire category theorem, some $X_t$ must be dense in some open box, and, hence, since it is semialgebraic, it must have dimension $n$.
\end{proof}

Condition (DP)(I) is known for our targeted example, namely $\la \overline \R, 2^\Z\ra$, and we provide precise references below. For the second example of Proposition \ref{prop-examples}, it can  easily be extracted from the literature, and we only provide a sketch (due to P. Hieronymi).


\begin{fact}\label{DPI}
Let $\WR$ be one of the structures in Proposition \ref{prop-examples}. Then \textup{(DP)(I)} holds.
\end{fact}
\begin{proof}
(1). This is \cite[Corollary 4.1.7]{tycho}. See also \cite[Section 8.6, Remark]{Mil1}. 

(2). As the proof of \cite[Corollary 4.1.7]{tycho} shows, (DP)(I) there follows from quantifier elimination and the fact that definable functions are given piecewise by terms. These two ingredients are also available in the case when $P$ is an iteration sequence (\cite[Theorem 1]{MT}), and hence the result follows easily.
\end{proof}

There are more expansions of $\overline \R$ where quantifier elimination, and hence (DP)(I), might also be possible to establish, see Questions \ref{qn-fast} and \ref{qn-sbd} below. \smallskip

We next focus on proving (DP)(II) in the examples of Proposition \ref{prop-examples}. This will follow from the next general statement.

\begin{prop}\label{prop-chunk} Assume $\WR$ is d-minimal. Then
$$\text{\textup{(DP)(I)} and \textup{(DIM)}\,\,\Rarr\,\, \textup{(DP)(II)}}.$$
\end{prop}

Before starting the proof of Proposition \ref{prop-chunk}, we need some prerequisites. Let $X\subseteq R^n$ be a definable set. Following Fornasiero \cite[Definition 4.1]{For}, we say that $X$ is an \emph{embedded manifold of dimension $m$} if for every $x\in X$, there is an open box $B\sub R^n$ containing $x$, such that, after permuting coordinates, $X\cap B$  is the graph of a continuous function $f:U\sub R^m\to V\sub R^{n-m}$, where $B=U\times V$.
The following result is stated in \cite{For}, but it can also be easily extracted from the work in \cite{tha}.

 \begin{fact}\label{fact-for}\cite[Fact 2.2]{For} Every definable set is a finite union of embedded manifolds.
\end{fact}

\begin{lemma}\label{LchunkY}
Let $X\sub R^n$ be a definable set of dimension $k$. Let $Y\sub X^n$ be a $k$-cell.  Assume that for every $y\in Y$,
\begin{align*}(*)_{y, X}\,\,\,\,\,&\text{ there is an open box $B\sub R^n$ containing $y$, such that $X\cap B$ is }\\ &\text{the graph of a continuous function $f:U\sub R^k\to V\sub R^{n-k}$, } \\ &\text{where $B=U\times V$.} \end{align*}
Then $Y$ is a semialgebraic chunk of $X$.
\end{lemma}
\begin{proof}
Since $Y$ is a $k$-cell, there is a projection $p: Y\to R^k$ onto some $k$ coordinates, which is injective. Let $\pi:R^n\to R^k$ be the projection onto the first $k$ coordinates. Then the map $\pi \circ p^{-1}  : p(Y) \to R^k$ is a continuous injective semialgebraic map, and hence, by \cite{johns}, its image $\pi(Y)\sub R^k$ is open.

Now take $y\in Y$, and let $B=U\times V$ and $f$ be as in our assumptions. Since $\pi(y)\in U\cap \pi(Y)$, there is an open box $U'\sub U\cap \pi(Y)$ containing $\pi(y)$. Let $B'=U'\times V$. So $y\in B'$. Moreover, for every $a\in U'$, there is $b\in R^{n-k}$ such that $(a, b)\in Y$. Since also $(a, b)\in X\cap B' =\Gamma (f_{\res B'})$, we have that $(X\cap B')_a$ is a singleton, namely, $\{b\}$. Hence, $X\cap B'\sub Y$, as needed. \end{proof}


\begin{proof}[Proof of Proposition \ref{prop-chunk}]
Let $X\sub R^n$ be a definable set of dimension $k$. By Fact \ref{fact-for}, $X$ is a finite union of embedded manifolds $X_1, \dots, X_l$. We split two (non-exclusive) cases:\\

\noindent\textbf{Case I:} 
$\dim \bigcap_i X_i =k$. In this case, by (DP)(I) and (DIM), there is a $k$-cell $Y\sub \bigcap_i X_i$. Take $y\in Y$. We claim that property $(*)_{y,X}$ from Lemma \ref{LchunkY} holds. Indeed, for every $i=1, \dots, l$, since $Y\sub X_i$ and $X_i$ is an embedded manifold, $(*)_{y,X_i}$ holds. Let $B_i\sub R^n$ be the corresponding open box containing $y$. Then $B=\cap B_i$ witnesses $(*)_{y,X}$. \\

\noindent\textbf{Case II:} For some $i\in \{1, \dots, l\}$,
$$\dim \left(X_i \sm \bigcup_{j\ne i} \overline{X_j}\right)=k.$$
In this case, by (DP)(I) and (DIM), there is a $k$-cell $Y\sub X_i \sm \bigcup_{j\ne i} \overline{X_j}$. Take $y\in Y$. We claim that property $(*)_{y,X}$ from Lemma \ref{LchunkY} holds. Indeed, since $Y\sub X_i$ and $X_i$ is an embedded manifold, $(*)_{y, X_i}$ holds. Namely, there is an open box $B\sub R^n$ containing $y$, such that $X_i\cap B$ is the graph of a continuous function $f:U\sub R^k\to V\sub R^{n-k}$, where $B=U\times V$. On the other hand, since $y\in Y$ and $Y\cap \bigcup_{j\ne i} \overline{X_j}=\es$, there is an open box $B'=U'\times V'\sub B$ containing $y$, such that $X\cap B'=X_i\cap B'$. Since $X_i\cap B'=\Gamma(f_{\res U'})$, we are done.
\end{proof}

\begin{proof}[Proof of Proposition \ref{prop-examples}]
By Remark \ref{rmk-thm}, Fact \ref{DPI} and Proposition \ref{prop-chunk}.
\end{proof}
We next apply our theorem to a structure beyond the d-minimal setting.

 \begin{cor}\label{cor-khani}
 Let $\WR=\la \overline \R, \R_{alg}, 2^\Z\ra$, where $\R_{alg}$ is the field of real algebraic numbers. Then every definable smooth function with open semialgebraic domain is semialgebraic.
 \end{cor}
\begin{proof}
Let $Y=\overline{\Gamma(f)}$. It is easy to see that $Y\cap (X\times \R)=\Gamma(f)$.
By \cite{khani}, every open definable set in $\WR$ is definable in $\la \overline \R, 2^\Z\ra$. We can thus apply Theorem \ref{main} to $f$.
\end{proof}

Observe that the above structure $\WR$ itself does not satisfy (DP). See Question \ref{qn-dense} below for further discussion.\smallskip

\subsection{Optimality}

The following example, due to P. Hieronymi, shows that if we remove the assumption that $\cal R$ is a reduct of $\overline R$ from Theorem \ref{main}, then the theorem fails. Note that the structure $\la\R_{an},e^{2\pi \Z}\ra$ is d-minimal, by \cite[Theorem 3.4.2]{Mil1}, and satisfies (DP)(I), by  \cite[Corollary 4.1.7]{tycho}. Thus, by Remark \ref{rmk-thm} and Proposition \ref{prop-chunk}, it satisfies all the assumptions of Theorem \ref{main}.

\begin{example}\label{exa-Philipp} The structure
$\WR =\la \overline \R_{an},e^{2\pi \Z}\ra$  defines new smooth functions with domain $\R$.
\end{example}
\begin{proof}
Let $\lambda:\R\rightarrow e^{2\pi\Z}$ be the function $\lam(x)=$ the biggest element of $e^{2\pi\Z}$ lower than $x$. Let $f:\R\to\R$ be the map $f(x)=\sin \log x$. Clearly, $f$ is not definable in the o-minimal $\R_{an}$, since its zero set is an infinite discrete set. We show that $f$ is definable in $\WR$. For every $x\in \R$, we have:
\[\begin{aligned}f(x)&=\sin\log\left(\lambda(x)\frac{x}{\lambda(x)}\right)\\&=\sin\log\lambda(x)+
\sin\log\left(\frac{x}{\lambda(x)}\right)\\&=\sin \log\left(\frac{x}{\lambda(x)}\right).
\end{aligned}\]
But $\frac{x}{\lambda(x)}\in [1,e^{2\pi}]$ and $\log([1,e^{2\pi}])=[0,2\pi]$. Therefore, $f$ is definable in $\la\R_{an},e^{2\pi \Z}\ra$.
\end{proof}

\begin{remark}\label{rmk-optimal} We make three further comments towards the optimality of our results:\smallskip

\noindent (1) The assumption that the domain of $f$ is open in Theorem \ref{main} is not essential. For example, one could consider definable smooth functions between semialgbebraic manifolds instead. Working via the chart maps, a generalization of our theorem in that setting eventually boils down to the current one.\smallskip

\noindent (2). The reader may wonder whether assuming only $R=\R$ (and not d-minimality) can yield more of our properties. Condition (DIM) fails: consider, for example, $\la \overline \R, P\ra$, where $P$ is the set of non-algebraic numbers. Let $f(t,s)=t+s:\R^2\to \R$. Then the definable family of singletons $\{(f(t,s)\}_{(t,s)\in P^2}$ violates (DIM).

\noindent (3). Assume $\WR$ is d-minimal and not o-minimal. We can define new $\mathcal{C}^n$-functions for all $n$. Indeed, it is easy to see that such a structure defines an infinite discrete set $A\sub R$. Let $s:A\rightarrow A$ be the successor function in $A$. For $n\in \N$, let $f_n:(0,\infty)\rightarrow R$ be given via, for every $t\in A$, $x\in [t,s(t)]\mapsto (x-t)^n(x-s(t))^n$. It is easy to see that $f_n$  is  $\mathcal{C}^{n-1}$, but not $\mathcal{C}^n$. It is also clear that $f_n$ is not semi-algebraic since its zero set is $A$.
\end{remark}

\subsection{Open questions}\label{sec-open}

We conclude with a list of open questions that arise naturally from this work.

\begin{question+} \textbf{Strong decomposition property.} Let us call a set $X\sub R^n$ a \emph{strongly embedded manifold} if  there is a semialgebraic family $\{Y_t\}_{t\in R^m}$ of subsets of $R^n$, and a definable set $S\sub R^m$ with $\dim S=0$, such that $X=\bigcup_{t\in S} Y_t$, and every $Y_t$ is a semialgebraic chunk of $X$.
\smallskip

{\em
  Assume  \textup{(DP)} and  \textup{(DIM)}. Is every definable set a finite union of strongly embedded manifolds?}
\end{question+}

\begin{question+}\label{qn-dense} \textbf{Necessary conditions.} Clearly,  \textup{(DP)} and  \textup{(DIM)} are not necessary conditions for Theorem \ref{main}. For example, in the dense pair $\la \overline \R, \R_{alg}\ra$, or in the example of Corollary \ref{cor-khani}, there are again no new smooth definable functions, but   (DP)(II) fails.
\smallskip

{\em
  Is there a reasonable way to relax Conditions  \textup{(DP)} and  \textup{(DIM)} into necessary and sufficient conditions for Theorem \ref{main}?}
\end{question+}

\begin{question+} \textbf{Smooth functions with arbitrary domains.} Here we relax the assumption on the domain of $f$ in Theorem \ref{main}.

  \smallskip

  {\em
Let $f: X\sub R^n\to R$ be a definable smooth function, with open connected domain $X$. Is it true that $f$ is the restriction to $X$ of a semialgebraic function?}
\end{question+}

The additional connectedness assumption is necessary. Indeed, let $\WR=\la \overline \R, 2^{\Z}\ra$, and $f:\bigcup_{t\in 2^\Z} (t, 2t)\to \R$ be given via: $f(x)=0$, if $x\in (t, 2t)$, for some $t\in 2^{2\Z}$, and $f(x)=1$ if $x\in (t, 2t)$ for some $t\in 2^{2\Z+1}$. Then $f$ is definable in $\WR$, with open domain, and it is smooth but not semialgebraic.

\begin{question+}\label{qn-fast} \textbf{Expansions by fast sequences.} As noted after Fact \ref{DPI}, it might be worth to explore whether quantifier elimination holds in further d-minimal expansions of the real field. We mention here one:
\smallskip

{\em
Assume $\WR$ is an expansion of the real field by a fast sequence \cite{Mil2}. Does it have quantifier elimination (and hence satisfies \textup{(DP)(I)}?}
\end{question+}

\begin{question+}\label{qn-sbd} \textbf{Expansions of semibounded structures.} By Example \ref{exa-Philipp}, a generalization of Theorem \ref{main}, when \cal R is replaced by an o-minimal expansion of it, fails. It is naturally to ask whether we can replace it by a proper reduct of it. Walsberg asked the following question:\smallskip

\emph{Let $\WR=\la  \R, <, \cdot_{\res (0,1)^2}, +, \Z\ra$. Is every definable smooth function $f:\R\to \R$ affine?}\smallskip

A proper reduct of \cal R is also is known in the literature as a \emph{semibounded} structure (see, for example, \cite{ed-sbd}). We recall from \cite{pet-sbd} that a \emph{long} interval $I\sub R$ is an interval on which no real closed field can be defined. We ask the following general question:\smallskip 

{\em
Let $\cal M=\la R, <, +, \dots\ra$ be a semibounded reduct of \cal R, and $\WR$ an expansion of $\cal M$ that satisfies \textup{(DP)} and \textup{(DIM)} with `semialgebraic' replaced by `definable in \cal M'. Let $f:I\sub R\to R$ be a definable smooth map, where $I\sub R$ is a long interval.  Is $f$ affine?}\smallskip

Currently, the assumptions of the last question are not known for the example $\WR=\la  \R, <, \cdot_{\res (0,1)^2}, +, \Z\ra$, although we expect them to be true. Indeed, (DP)(I) would follow from a quantifier elimination result, which perhaps can be proved following the same strategy as in \cite[Appendix]{Mi-IVT} for $\la \Q, <, +, \Z\ra$. Note, however, that Proposition \ref{prop-chunk} still holds with the adapted assumptions and an analogous proof.


\end{question+}


\begin{thebibliography}{999999}


\bibitem{bh} G. Boxall, P. Hieronymi, {\em Expansions which introduce no new open sets}, Journal of Symbolic Logic, (1) 77 (2012) 111--121.



\bibitem{BCR} J. Bochnak, M. Coste, M.-F. Roy, {\sc Real Algebraic Geometry}, Ergebnisse der Mathematik und ihrer Grenzebiete (3. Folge), vol. 36, Springer-Verlag, 1998.

\bibitem{VDD} L. van den Dries, {\sc Tame topology and o-minimal structures}, Cambridge University Press, Cambridge, 1998.

\bibitem{VDD1} L. van den Dries, {\em The field of reals with a predicate for the power of two}, Manuscripta mathematica 54 (1985), 187--195.


\bibitem{ed-sbd} M. Edmundo, {\em Structure theorems for o-minimal expansions
of groups}, Ann. Pure Appl. Logic 102 (2000), 159-181.




\bibitem{For} A. Fornasiero, {\em Groups and rings definable in d-minimal structures}, eprint arXiv:1205.4177 (2012).

\bibitem{Mil2} H. Friedman, C. Miller, {\em Expansions of o-minimal structures by fast sequences}, Journal of Symbolic Logic 70 (2005), 410--418.



\bibitem{johns} J. Johns, \emph{An Open Mapping Theorem for O-Minimal Structures}, Journal of Symbolic Logic 66 (2001), 1817--1820.

\bibitem{khani} M. Khani, \emph{The field of reals with a predicate for the real algebraic numbers and a predicate for the integer powers of two}, Archive for Mathematical Logic  54 (2015), 885--898.


    \bibitem{mpp} D. Marker, Y. Peterzil, and A. Pillay, {\em
Additive Reducts of Real Closed Fields},
J. Symbolic Logic
 57 (1992), 109--117.



\bibitem{Mi-IVT} C. Miller, {\em Expansions of Dense Linear Orders with the Intermediate Value Property},  Journal of Symbolic Logic 66 (2001),  1783--1790.

\bibitem{Mil1}C. Miller, {\em Tameness in expansion of the real field}, Logic Colloquium '01, Lecture Notes in Logic, Vol. 20, Association of Symbolic Logic, Urbana 11, pp 281-316, MR 2143901, 2005.


\bibitem{MT}C. Miller, J. Tyne, {\em
Expansions of o-Minimal Structures
by Iteration Sequences},     Notre Dame Journal of Formal Logic 47 (2006), 93--99..


\bibitem {pet-sbd} Y. Peterzil, {\em Returning to semi-bounded sets}, J. Symbolic Logic  74  (2009), 597-617.

\bibitem{pss} A. Pillay, P. Scowcroft, and C. Steinhorn, {\em Between Groups and rings}, Rocky Mountain J. Math. 19 (1989), 871--886.

\bibitem{tha} A. Thamrongthanayalak, {\em Michael's selection theorem in d-minimal expansions of the real field}, Preprint, 2017.


\bibitem{tycho} M. Tychonievich, {\em Tameness results for expansions of the real field by groups}, Ph.D. Thesis, Ohio State University, 2013.


\end{thebibliography}
\end{document}